\documentclass[12pt,a4paper]{amsart}
\usepackage{amsmath}
\usepackage{amsfonts}
\usepackage{color}
\font\bb=msbm7 at 11pt

\def\cqfd{
{\hfill
\kern 6pt\penalty 500
\raise -1pt\hbox{\vrule\vbox to 5pt{\hrule width 4pt
\vfill\hrule}\vrule}}
\break}
\def\deg{\mathop{\rm deg}\nolimits}

\def \F {\hbox{\bb F}}

\def\f#1{\hbox{\bb F}_{#1}}

\def \A {\hbox{\bb A}}
\def \P {\hbox{\bb P}}
\long\def\eg#1{{\color{green}#1}}
\long\def\eg#1{}
\def\no{n$^\circ$}

\let\cj=\overline

\newtheorem{theorem}{Theorem}
\newtheorem{proposition}[theorem]{Proposition}
\newtheorem{lemma}[theorem]{Lemma}

\begin{document}

\date{\today}

\title{Differentially 4-uniform functions}
\author{Yves Aubry and Fran\c cois Rodier}
\address{Institut de Math\'ematiques de Toulon, Universit\'e du Sud Toulon-Var, France\\
and \\
Institut de Math\'ematiques de Luminy, Marseille, France}

\email{yves.aubry@univ-tln.fr and rodier@iml.univ-mrs.fr}



\begin{abstract}
We give a geometric characterization of vectorial Boolean functions with differential uniformity $\leq 4$. This enables us to give a necessary condition on the degree of the base field for a function of degree $2^r-1$ to be differentially 4-uniform.
\end{abstract}

\maketitle

\section{Introduction}
We are interested in {vectorial} {Boolean} functions from the $\F_2$-vectorial space ${\F}_2^m$ to itself  in $m$ variables, viewed as polynomial functions $f:{\F}_{2^m} \longrightarrow {\F}_{2^m}$ over the field ${\F}_{2^m}$ in one variable of degree at most $2^{m}-1$. 
For a function  $f:{\F}_{2^m} \longrightarrow {\F}_{2^m}$, we consider, after K.  Nyberg (see \cite{ny}),  {its differential uniformity}
$$\delta(f)=\max_{\alpha\not=0,\beta}\sharp\{x\in{\F}_{2^m} \mid f(x+\alpha)+f(x)=\beta\}.$$
This is  clearly a strictly positive even integer.

Functions $f$ with small $\delta(f)$ have applications in  cryptography {(see \cite{ny})}. Such functions with $\delta(f)=2$ are called almost perfect nonlinear (APN) and have been  extensively studied: see  \cite{ny}  and \cite{ccz} for the genesis of the topic and more recently  \cite{BCCL} and \cite{BCL}  for a synthesis of open problems; see also \cite{bcp} for new constructions and \cite{Vol} for a geometric point of view of differential uniformity.

 Functions with $\delta(f)=4$ are also useful; for example the function $x\longmapsto x^{-1}$, which is used in the AES algorithm over the field $\f{2^8}$,  has differential uniformity 4 on $\f{2^m}$ for any even $m$.
Some results on these functions have been collected by C. Bracken and G. Leander \cite{bl1,bl2}.

We consider here the class of  functions $f$ such that $\delta(f)\leq 4$, called differentially 4-uniform functions. We will show that for polynomial   functions $f$ of degree $d=2^r-1$ such that $\delta(f)\leq 4$ on the field ${\F} _{2^m}$, the number $m$ is bounded by an expression depending on $d$. The second author demonstrated the same bound in the case of APN functions \cite{Ro1, Ro2}.
The principle of the method we apply here was already used by H. Janwa et al. \cite{J-W} to study cyclic codes and by A. Canteaut \cite{ca} to show that certain power functions could not be APN when the exponent is too large.

Henceforth we fix $q=2^m$.

In order to simplify our study of  such functions, let us recall the following elementary results on differential uniformity; the
proofs are straightforward:

\begin{proposition}
\label{reduction}
{\sl (i)} Adding a $q$-affine polynomial (i.e.  a polynomial whose monomials are of degree 0 or a power of 2) to a function $f$ does not change $\delta(f)$.

{\sl (ii)} For all $a$, $b$ and $c$ in $\f q$, such that $a\ne0$ and $c\ne0$ we have
$$\delta(cf(ax+b))=\delta(f).$$

{\sl (iii)} One has $\delta(f^2)=\delta(f).$
 \end{proposition}

Hence, without loss of generality, 
from now on we can assume that  $f$ is a polynomial mapping from $ \f {q} $ to itself which  has neither terms of degree a power of 2 nor a constant term, and which has at least one term of odd degree.

To any function $f:{\F}_{q} \longrightarrow {\F}_{q}$, we associate the polynomial
$$f(x)+f(y)+f(z)+f(x+y+z).$$
Since this polynomial  is clearly divisible by 
$$(x+y)(x+z)(y+z),$$
we can consider the polynomial 
$${P_f}(x,y,z):=\frac{f(x)+f(y)+f(z)+f(x+y+z)}{(x+y)(x+z)(y+z)}$$
which has degree  $\deg(f)-3$ if $\deg(f)$ is not  a power of 2.

%%%%%%%%%%%%%%%%%%%%%%%%%%%%%%%%%%%%%%%%%%%%%%%%%%%%

\section{A characterization of  functions with $\delta\le4$}\label{characterization}

We will give, as in \cite{Ro1}, a geometric criterion for a  function to have $\delta\le4$.
We consider in this section the algebraic set $X$ defined by the elements $(x,y,z,t)$ in the affine space $\A^4(\overline{\F}_q)$ such that
$${P_f}(x,y,z)={P_f}(x,y,t)=0.$$

We set also
$V$ the hypersurface of the affine space $\A^4(\overline{\F}_q)$ defined by 
\begin{equation}\label{7hyperplans}
(x+y)(x+z)(x+t)(y+z)(y+t)(z+t)(x+y+z+t)=0.
\end{equation}

The hypersurface $V$ is the union of the seven hyperplanes $H_1$, \ldots, $H_7$ defined respectively by the equations $x+y=0$, \ldots, $x+y+z+t=0$.

We begin with a simple lemma:

\begin{lemma}
The following two properties are equivalent:

(i) there exist 6 distinct elements $x_0, x_1, x_2, x_3, x_4, x_5$ in ${\F}_{q}$ such that
$$\begin{cases}
x_0+x_1=\alpha,\ \ f(x_0)+f(x_1)=\beta\\
x_2+x_3=\alpha,\ \ f(x_2)+f(x_3)=\beta\\
x_4+x_5=\alpha,\ \ f(x_4)+f(x_5)=\beta
\end{cases}$$

(ii) there exist 4 distinct elements $x_0, x_1, x_2, x_4$  in ${\F}_{q}$  such that $x_0+x_1+ x_2+ x_4\ne0$
 and such that
$$\begin{cases}
f(x_0)+f(x_1)+f(x_2)+f(x_0+x_1+x_2)=0\\
f(x_0)+f(x_1)+f(x_4)+f(x_0+x_1+x_4)=0.\\
\end{cases}$$
\end{lemma}

\begin{proof}
Suppose that (i) is true.
Then we have
$x_0+x_1+x_2=\alpha+x_2=x_3$ and so
$f(x_0)+f(x_1)+f(x_2)+f(x_0+x_1+x_2)=f(x_0)+f(x_1)+f(x_2)+f(x_3)=0.$
The second equation holds true in the same way. Finally, we have $x_0+x_1+ x_2+ x_4=x_3+x_4\ne0.$

Conversely, let us set
  $\alpha=x_0+x_1$, $\beta=f(x_0)+f(x_1)$ and $x_3=\alpha+x_2=x_0+x_1+x_2$.
Then
$f(x_2)+f(x_3)=f(x_2)+f(x_0+x_1+x_2)=f(x_0)+f(x_1)=\beta$.
Furthermore, we have $x_3\ne x_0$ because $x_1\ne x_2$ and we have
$x_3\ne x_1$ since otherwise we would have  $x_2=\alpha+x_3=\alpha+x_1=x_0$.

Setting
 $x_5=\alpha+x_4=x_0+x_1+x_4$
we have
$f(x_4)+f(x_5)=f(x_4)+f(x_0+x_1+x_4)=f(x_0)+f(x_1)=\beta$.
We have $x_3\ne x_4$ since otherwise we would have $0=x_3+ x_4=x_0+x_1+ x_2+ x_4$ which is not the case by hypothesis.

Finally $x_3\ne x_5$ since otherwise we would have $ x_2= x_4$, and so
all the six elements $x_0, x_1, x_2, x_3, x_4, x_5$  are different.
\end{proof}

We can now state a geometric characterization of differentially 4-uniform functions:

\begin{theorem}\label{a2pn}
The differential uniformity of a function $f:{\F}_{q} \longrightarrow {\F}_{q}$  is not larger than 4 if and only if:
$$X({\F}_q)\subset V$$
where $X({\F}_q)$ denotes the set of rational points over ${\F}_q$ of $X$.
\end{theorem}

\begin{proof}
The differential uniformity is not larger than 4 if and only if for any $\alpha\in{\F}_{q}^{\ast}$ and any $\beta\in{\F}_{q}$, the equation 
$$f(x+\alpha)+f(x)=\beta$$
has at most 4 solutions, that is to say
$$\sharp\{x\in{\F}_{q} \vert f(x)+f(y)=\beta,\ \  x+y=\alpha\}\leq 4.$$
But this is equivalent to saying that we cannot find 6 distinct elements $x_0, x_1, x_2, x_3, x_4, x_5$ in ${\F}_{q}$ such that
$$\begin{cases}
x_0+x_1=\alpha,\ \ f(x_0)+f(x_1)=\beta\\
x_2+x_3=\alpha,\ \ f(x_2)+f(x_3)=\beta\\
x_4+x_5=\alpha,\ \ f(x_4)+f(x_5)=\beta.
\end{cases}$$
By the previous lemma, this is equivalent  to saying that we cannot find 4 distinct elements $x_0, x_1, x_2, x_4$ in ${\F}_{q}$
 such that $x_0+x_1+ x_2+ x_4\ne0$ and such that
$$\begin{cases}
f(x_0)+f(x_1)+f(x_2)+f(x_0+x_1+x_2)=0\\
f(x_0)+f(x_1)+f(x_4)+f(x_0+x_1+x_4)=0.\\
\end{cases}$$
But this  can be reformulated by saying that the rational points over $\F_q$ of the variety $X$  are contained in the variety $V$, that is to say
$X({\F}_{q})\subset V$.

\end{proof}
%==============

%%%%%%%%%%%%
%%%%%%%%%%%%
%%%%%%%%%%
%%%%%%%%%%%

\section{Monomial  functions with $\delta\le4$}\label{monomial}

If the function $f$ is a monomial of degree $d>3$:
$$f(x)=x^d$$
then the polynomials ${P_f}(x,y,z)$ and ${P_f}(x,y,t)$ are homogeneous polynomials and we can consider 
the intersection $X$ of the projective cones $S_1$ and $S_2$ of dimension 2 
defined respectively by
${P_f}(x,y,z)=0$ and ${P_f}(x,y,t)=0$ with projective coordinates $(x:y:z:t)$ in the projective space $\P^3(\overline{\F}_{q})$.

Even if  $X$ is now a projective algebraic subset of the projective space $\P^3(\overline{\F}_{q})$, 
Theorem \ref{a2pn} tells us also that:
$$\delta(f)\leq 4 \ \ {\rm if}\ {\rm and}\ {\rm only}\ {\rm if}\ \ X({\F}_q)\subset V,$$
where $V$ is the hypersurface of $\P^3(\overline{\F}_{q})$ defined by Equation (\ref{7hyperplans}).

Indeed,  the algebraic sets $X$ and $V$ in this section are closely related to  but not equal to the sets $X$ and $V$ of the previous section. The set $X$ of this section (resp. $V$) is the set of lines through the origin of the set $X$ (resp. $V$) of the previous section which is invariant under homotheties with center the origin.
For convenience, we keep the same notations.
\begin{lemma}
The projective algebraic set $X$ has dimension 1, i.e.  it is a projective curve.
\end{lemma}
\begin{proof}
We have to show that the projective surfaces $S_1$ and $S_2$ do not have  common irreducible components. Since $S_1$ and $S_2$ are two cones, it is enough to prove that the vertex of one of the cones doesn't lie in the other cone. The coordinates of the vertex of the cone $S_2$ is $(0:0:1:0)$. To show that it doesn't lie in $S_1$, 
we will prove that $P_f(0:0:1:0)\ne0$.
 Indeed, $S_1$ is defined by the polynomial
$${P_f}(x,y,z)=\frac{x^d+y^d+z^d+(x+y+z)^d}{(x+y)(x+z)(y+z)}\cdot$$
Setting $x+y=u$, we obtain:

$${P_f}(x,y,z)=\frac{x^{d}+(x+u)^d+z^{d}+(u+z)^d}{u(x+z)(x+u+z)},$$
which gives
$${P_f}(x,y,z)=\frac{x^{d-1}+z^{d-1}+uQ(x,z)}{(x+z)(x+u+z)},$$
where $Q$ is some polynomial in $x$ and $z$.
This expression takes the value $1$ at the point $(0:0:1:0)$.
\end{proof}

Now we know that $X$ is a projective curve in $\P^3(\overline{\F}_{q})$, and in order to estimate its number of rational points over ${\F}_{q}$, we must determine its irreducibility. We will prove that the curve
$C_7$, defined as the intersection of $S_2$ with the projective plane $H_7$ of equation 
$x+y+z+t=0$, is an absolutely irreducible component of $X$, and hence that $X$ is reducible.

\begin{proposition}
\label{intercourbe}
The intersection of the curve $X$ with the plane $H_7$ with the equation $x+y+z+t=0$ is equal to the curve  $C_7:=S_2\cap H_7$.
\end{proposition}

\begin{proof}
Since $X=S_1\cap S_2$, it is enough to prove that $C_7\subset S_1$.
Since $t=x+y+z$ the points of intersection of the cone $S_2$ with the plane $x+y+z+t=0$ satisfy:
\begin{eqnarray*}
0={P_f}(x,y,t)
&=&\frac{x^d+y^d+t^d+(x+y+t)^d}{(x+y)(x+t)(y+t)}\\
&=&\frac{x^d+y^d+(x+y+z)^d+z^d}{(x+y)(y+z)(x+z)}\\
&=&{P_f}(x,y,z),
\end{eqnarray*}
so they belong to $S_1$.
\end{proof}

\begin{proposition}
\label{isocourbe}
The projective plane curve $C_7$ is isomorphic to the projective plane curve $C$ with equation
$${P_f}(x,y,z)=\frac{x^d+y^d+z^d+(x+y+z)^d}{(x+y)(x+z)(y+z)}=0.$$
\end{proposition}

\begin{proof}
The projection from the vertex of the cone $S_1$ defines an isomorphism of the projective plane $H_7$ with equation
$x+y+z+t=0$ onto the plane  with equation $t=0$, and it maps $C_7$ onto the curve $C$ with equation
${P_f}(x,y,z)=0.$
\end{proof}

\begin{proposition}\label{courbeinterV}
Let $\mathcal C$ be a plane curve of degree $\deg({\mathcal C})$ and which is not contained in $V$.
Then:
$$\sharp ({\mathcal C}\cap V)({\F}_q)\leq 7\deg({\mathcal C}).$$
\end{proposition}

\begin{proof}
The variety $V$ is the union of seven projective planes. Each plane cannot contain more than $\deg({\mathcal C})$ points, therefore $V$ contains at most $7\deg({\mathcal C})$ rational points in $\mathcal C$.
\end{proof}

In order to get a lower bound for the number of rational points over $\F_q$ on the curve $C$, hence on the curve $X$, we need to know if $C$ is absolutely irreducible or not. This question has been discussed by H. Janwa, G. McGuire and R. M. Wilson in \cite{J-M-W} and  very recently by F. Hernando and G. McGuire in \cite{H-M}.

\begin{proposition}\label{component}
If $d=2^{r}-1$ with $r\geq 3$, then the projective curve $X$ has an absolutely irreducible component $C'$ defined over $\F_2$ in the plane $x+z+t=0$ and this component $C'$ is isomorphic to the curve $C$.
\end{proposition}

\begin{proof}
One checks that the intersection of the cone $S_1$ with the plane $x+z+t=0$ is the same as  the intersection of the cone $S_2$ with that plane. Hence one can show, as in Proposition  \ref{isocourbe}, that the  intersection of the curve $X$ with  the plane $x+z+t=0$ is isomorphic to the curve $C$.
Furthermore, it is proved in \cite{J-M-W} that  the curve $C$ is absolutely irreducible since, $\deg(C)=2^r-1\equiv 3\pmod 4$.
\end{proof}

\noindent
Hence we can state

\begin{theorem}
\label{borne1}
Consider the function $f:{\F}_{q} \longrightarrow {\F}_{q}$ defined by $f(x)=x^d$ with $d=2^r-1$ and $r\ge3$. 
If  $5\le d<q^{1/4}+4.6$~, then $f$ has differential uniformity strictly greater than 4. 
\end{theorem}

\medskip

\begin{proof}

The curve $C'$ is an absolutely irreducible plane curve of arithmetic genus  $\pi_{C'}=(d-4)(d-5)/2$.
According to \cite{A-P-1} (see also \cite{A-P-2} for a more general statement), the number of rational points of the (possibly singular) absolutely irreducible curve $C'$ satisfies
  $$| \#C'({\F}_q)-(q+1)| \leq 2\pi_{C'} q^{1/2}.$$
      Hence
         $$ \#C'({\F}_{q})\ge q+1- 2\pi_{C'} q^{1/2}.$$

      The maximum number of rational points on the curve $C'$ on the surface $V$ is $7(d-3)$ by Proposition \ref{courbeinterV}. If
        $q+1- 2\pi_C q^{1/2}> 7(d-3)$, then
        $C'({\F}_{q})\not\subset V$, therefore $X({\F}_{q})\not\subset V$, and
          $\delta(f)>4$   by Theorem \ref{a2pn}.
But this condition is equivalent to
$$ q- 2\pi_{C'} q^{1/2} - 7(d-3)+1> 0.$$
 The condition is satisfied when  
$$q^{1/2}>\pi_{C'} + \sqrt{7(d-3)-1 + \pi_{C'}^2}$$
hence when 
$$q\ge d^4-18 d^3+121 d^2-348 d+362$$
or
$$5\le d<q^{1/4}+4.6.$$

\end{proof}

%%%%%%%%%%%%%%%%%%%%%%
%%%%%%%%%%%%%%%%%%%%%%
%%%%%%%%%%%%%%%%%%%%%%%

\section{Polynomials  functions with $\delta\le4$}

If the function $f$ is a polynomial of one variable with coefficients in ${\F}_q$ of degree $d>3$, we consider again as in section \ref{monomial}
the intersection $X$ of  $S_1$ and $S_2$, which are now cylinders in the affine space $\A^4(\overline{\F}_q)$ with equations respectively
${P_f}(x,y,z)=0$ and ${P_f}(x,y,t)=0$ and which are of dimension 3 as affine varieties.

\eg{If the polynomial $\check{P_f}(x,y,z)$ is absolutely irreducible, the projective variety $X:=\overline{\check{X}}$ is a   projective surface in ${\P}^4(\overline{\F}_{q})$.
}

\begin{lemma}
The algebraic set $X$ has dimension 2,  i.e.  it is an affine surface. Moreover, it has  degree $(d-3)^2$.
\end{lemma}
\begin{proof}
%$$\overline{X}=Z_{{\P}^4}(P_f^{\ast}(x,y,z,u))\cap Z_{{\P}^4}(P_f^{\ast}(x,y,t,u))$$

We have to show that the hypersurfaces $S_1$ and $S_2$ do not have a common irreducible component. Since these hypersurfaces are two cylinders, 
it is enough to prove that the polynomial defining $S_1$ does not vanish on the whole of a straight line $(x_0,y_0,z,t_0)$ where $x_0,y_0, t_0$ are fixed and satisfy $P_f(x_0,y_0,t_0)=0$.   Indeed, $S_1$ is defined by the polynomial ${P_f}(x,y,z)$, which takes the value
$${P_f}(x_0,y_0,z)=\frac{f(x_0)+f(y_0)+f(z)+f(x_0+y_0+z)}{(x_0+y_0)(x_0+z)(y_0+z)}$$
at the point $(x_0,y_0,z,t_0)$.
If we set $x_0+y_0=s_0$,  the homogeneous term of degree $d_i$ in ${P_f}(x,y,z)$ becomes
$$\frac{d_i(x_0^{d_i-1}+z^{d_i-1})+s_0Q_i(x_0,z)}{(z+s_0+x_0)(z+x_0)}$$
where $Q_i$ is a polynomial in $x_0$ and $z$ of degree $d_i-2$.
If $d_i$ is odd, the numerator of this term is of degree $d_i-2$, and hence does not vanish, so it is the same for the polynomial ${P_f}(x_0,y_0,z)$. Hence, $X$ has dimension 2.
Moreover, $X$ is the intersection of two hypersurfaces of degree $d-3$, thus it has degree $(d-3)^2$.
\end{proof}

The surface ${{X}}$ is reducible.
Let $X=\bigcup_{i}X_i$ be its decomposition in absolutely irreducible components.

We embed the affine surface $X$ into a projective space $\P^4(\overline{\F_q})$ with homogeneous coordinates $(x:y:z:t:u)$. Consider the hyperplane at infinity $H_{\infty}$ defined by the equation $u=0$ and let
$X_\infty$ be the intersection of the projective closure $\overline X$ of $X$ with $H_{\infty}$.
Then $X_{\infty}$ is the intersection of two surfaces in this hyperplane, which are respectively the intersections  $S_{1,\infty}$ and $S_{2,\infty}$ of the cylinders $S_1$ and $S_2$ with that hyperplane. 
The homogeneous equations of  $S_{1,\infty}$ and $S_{2,\infty}$ are 
$${P_{x^d}}(x,y,z)=\frac{x^d+y^d+z^d+(x+y+z)^d}{(x+y)(x+z)(y+z)}$$
and
$${P_{x^d}}(x,y,t)=\frac{x^d+y^d+t^d+(x+y+t)^d}{(x+y)(x+t)(y+t)}\cdot$$

By Proposition \ref{component}, the intersection of the curve $X_\infty$ with the plane $x+z+t=0$ (inside the hyperplane at infinity) is an absolutely irreducible component $C'$ of the curve $X_\infty$ of multiplicity 1, defined over $\f2$. So the only absolutely irreducible component of $\overline X$, say $\overline{X}_1$, which contains $C'$ is defined over ${\F}_q$.

\begin{proposition}
\label{maxSinterV}
Let $\mathcal X$ be an absolutely irreducible projective surface of degree $>1$. Then
the maximum number of rational points on  $\mathcal X$ which are contained in  the hypersurface ${\overline V}\cup H_{\infty}$ is

$$\sharp({\mathcal X}\cap({\overline V}\cup H_{\infty}))\leq 8(\deg(\mathcal X)q+1).$$
\end{proposition}

\begin{proof}
As $\deg(\mathcal X)>1$,  the surface $\mathcal X$ is not contained in any hyperplane. Thus, a hyperplane section of $\mathcal X$ is a curve of degree $\deg(\mathcal X)$. Using the bound on the maximum number of rational points on a general hypersurface of given degree proved by Serre in \cite{se}, we get the result.
\end{proof}
%\er{On a besoin ici du fait que $\check{P_f^*}(x,y,z,u)$ est sans racines doubles}

\begin{theorem}
\label{lawe}
Consider a function $f:{\F}_{q} \longrightarrow {\F}_{q}$ of degree $d=2^r-1$ with $r\geq 3$.
If  $31\le d< q^{1/8}+2 $,  then $\delta(f)>4$. 
For $d<31$, we get $\delta(f)>4$ for $d=7$ and $m\geq 22$ and also if $d=15$ and $m\geq 30$. 
\end{theorem}

\begin{proof}
From an improvement of a result of  S. Lang and A. Weil \cite{lw} proved by S. Ghorpade and G. Lachaud \cite[section 11]{G-L}, we deduce
        \begin{eqnarray*}
|\#\cj X_1(\f{q})-q^2-q-1|&\le &((d-3)^2-1)((d-3)^2-2) q^{3/2}+36(2d-3)^5q\\
&\le &(d-3)^4 q^{3/2}+36(2d-3)^5q.
\end{eqnarray*}
        Hence
         $$ \#\cj X_1(\f{q})\ge q^2+q+1-  (d-3)^4 q^{3/2}-36(2d-3)^5q.$$
Therefore, if 
$$q^2+q+1- (d-3)^4 q^{3/2}-36(2d-3)^5q> 8((d-3)q+1),$$
then
         $\#\cj X(\f{q})\ge \#\cj X_1(\f{q})>8((d-3)q+1)$, 
         and hence $\cj X_1(\f{q})\not\subset {\overline V}\cup H_{\infty}$
        by Proposition  \ref{maxSinterV}.
         As $X$ is the set of affine points of the projective surface $\cj X$, we deduce that $X(\F_q)\not\subset V$
         and so the differential uniformity of $f$ is at least 6 from Theorem \ref{a2pn}.
         This condition can be written
         $$q- (d-3)^4 q^{1/2} -36(2d-3)^5-8(d-3)> 0.$$
This condition is satisfied when
$$q^{1/2}>d^4-12 d^3+54 d^2+1044 d+5265+25920/d$$ if $d\ge2$,
or 
$d<q^{1/8} +2$ if $d\ge 31$.
\end{proof}

\end{document}